\documentclass[12pt]{article}

\RequirePackage[OT1]{fontenc}
\RequirePackage{amsthm,amsmath}
\RequirePackage[numbers]{natbib}
\RequirePackage[colorlinks,citecolor=blue,urlcolor=blue]{hyperref}

\usepackage{amsmath}
\usepackage{graphicx,psfrag,epsf}
\usepackage{enumerate}
\usepackage{natbib}
\usepackage{url} 
\usepackage{hyperref}
\usepackage{amsfonts}
\usepackage{listings}
\usepackage{hyperref}
\usepackage{lipsum}
\usepackage{amsthm,xpatch}
\usepackage[page,toc,titletoc,title]{appendix}
\usepackage{tocloft}
\usepackage{mathtools}

\usepackage{xcolor}

\newtheorem{theorem}{Theorem}
\newtheorem{corollary}{Corollary}[theorem]
\newtheorem{lemma}[theorem]{Lemma}
\newtheorem{proposition}[theorem]{Proposition}

\newcommand{\blind}{1}

\begin{document}

\newcounter{proofpart}

\def\spacingset#1{\renewcommand{\baselinestretch}%
{#1}\small\normalsize} \spacingset{1}


\if1\blind
{
  \title{\bf A proof of Blackwell's renewal theorem by mapping to integers}
  \author{Rohit Pandey}
  \maketitle
} \fi

\if0\blind
{
  \bigskip
  \bigskip
  \bigskip
  \begin{center}
    {\LARGE\bf A simple proof of Blackwell's renewal theorem by mapping to integers}
\end{center}
  \medskip
} \fi

\begin{abstract}
This paper presents a new proof of the renewal theorem by bijecting a general point process to a deterministic one (where the time between events is always fixed). It also provides insight into the workings of the renewal theorem.
\end{abstract}

\section{Counting events}\label{sectn:count}
Imagine you go to a bus stop at some random time in the day, not with the intention of boarding any bus, but simply counting them. The time between successive bus arrivals is a random variable, but we know the expected value of this random variable is $t$. If you wait at the bus stop for an interval $u$, what should be the expected number of buses you count? Intuition says that we should count on average $\frac{u}{t}$ buses (ex. if buses take on average $10$ minutes to arrive and you wait at the stop for $30$ minutes, you'll see on average 3 buses).

This obviously doesn't always hold. If the buses arrive exactly $t=10$ minutes apart (we'll call such a process a deterministic point process) and we start our observation period at time $0$ and make its length $u=9$ minutes, the average number of arrivals we'll observe won't be $\frac{9}{10}$, it'll be zero (since we always miss the first bus which arrives at $10$ minutes from the start of the process). 

But this is, `cheating' since the interval is strategically placed at a point where the process just ``renewed" itself.
We need to start our interval in a way that it doesn't ``prefer" any of the renewals of the process.

If we manage to do this, it turns out to be a very robust result. Simulations show that it holds for a variety of inter-arrival distributions, for both renewal as well as more general processes. For instance, see listing-\ref{lst:simulator} in appendix \ref{py:code} and also the page \href{https://gist.github.com/ryu577/662f9cb593d40920e161cfef3eba0244}{here} with some Python code demonstrating this for various inter-arrival distributions.

This is a well known result in renewal theory called the renewal theorem or Blackwell's renewal theorem. See for example section 11.1 of \cite{feller} and theorem 4.6.2 of \cite{ocw_renewal}. 

This theorem is proved via a method called coupling in \cite{blackwell_proof}. In \cite{non_iid}, this is extended to processes where the inter-arrival times might not be identically distributed and correlated (which should enable one to model any arbitrary point process).

This paper will provide a new proof for this result which relies on a very simple argument that maps any point process to a deterministic point process (one where events happen every fixed, deterministic interval, $t$; ex: buses arriving at the hour, every hour or your favorite TV show airing at 7:00 PM every Saturday). Note also that a deterministic process can be re-scaled by dividing the entire time-line (all measured times) by $t$ (which is equivalent to changing the unit in which time is measured so $t$ becomes $1$ unit), causing the events to happen at every positive integer.

Let's start by defining some terminology (mostly inspired by \cite{{ocw_renewal}}) with regard to point processes. Then, we'll describe the family of point processes the deterministic point process belongs to since it plays a central role in our paper and finally, state Blackwell's renewal theorem formally.

\subsection{Terminology of point processes}\label{terminology}
Going back to our example with the buses, the $i$-th bus arrival has the time since the previous ($i-1$-th arrival) distributed as $T_i$. In general, the $T_i$'s can be correlated. If we consider them to be i.i.d, we get a renewal process. Since they describe the time until the next event, they can't be negative and are hence supported on $(0,\infty)$.

The expected value of the $T_i$'s is given by $\Bbb E(T_i) = t$. We can separate the $T_i$ into deterministic and random parts: $T_i = t+\epsilon_i$. Which means the $\epsilon_i$ are zero-mean random variables supported on $(-t, \infty)$ (since the $T_i$ can't be negative). 

The absolute time at which the $i$-th event happens is given by: $S_i = \sum\limits_{k=1}^i T_i$. Further, we arrive at the bus stop and start counting buses at time $u_1$ and continue counting until some pre-defined time, $u_2$. The length of the observation interval is $u=u_2-u_1$. This is all depicted in figure \ref{fig:pt_process_v0}. 

\begin{figure}
  \includegraphics[width=0.8\linewidth]{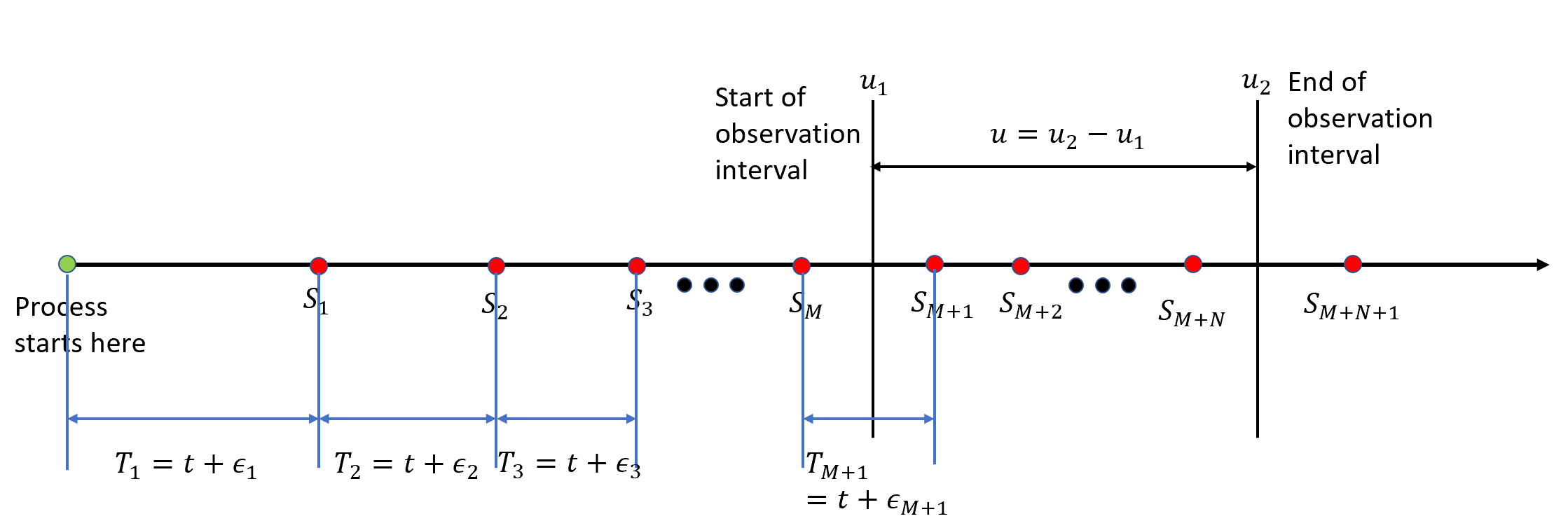}
  \caption{A basic point process with observation interval. The point-in-time events are shown as red dots on the time-line.}
  \label{fig:pt_process_v0}
\end{figure}

The number of point in time events from the point process happening from the start to some time, $s$ is another random variable and is denoted by $N(s)$. Going back to our example with the buses, the number of events from the start of the process to when we start observing it, $u_1$ is given by $N(u_1)$ and that from the start of the process to the end of our observation window is given by $N(u_2)$. Since $u_1$ is time with special significance to us, we denote $M \coloneqq N(u_1)$ and the number of events falling into our observation interval, $N(u_1 \to u_2) \coloneqq N(u_2)-N(u_1) \coloneqq N$. So, the absolute time stamps at which events happen inside our observation interval are $S_{M+1}, S_{M+2} \dots S_{M+N}$. Refer again to figure \ref{fig:pt_process_v0} for a visual depiction of all of this.

Since $N(s)$ is a random variable, we can think also of its expectation: $\mu(s) \coloneqq \Bbb E (N(s))$. The average number of points lying inside the interval between $u_1$ and $u_2$ then becomes: $\mu(u_1 \to u_2) \coloneqq \Bbb E (N(u_1 \to u_2)) =\Bbb E(N(u_2)-N(u_1)) = \mu(u_2)-\mu(u_1)$. The result this paper revolves around estimates this very quantity. 

Finally, in this paper, $U$ will always mean a uniform random number over $(0,1]$, $\vartheta$ will always mean some very large number and $V$ will mean a uniform random number over $(0,\vartheta]$ (so $V=\vartheta U$).

\subsection{Arithmetic point processes}
Imagine a point process where events can happen only at integer time stamps. And an obvious generalization of such a process happens by scaling its whole time line by some scalar, $\lambda$ (which can be interpreted simply as measuring time in different units). Now, the events happen only at multiples of the scalar, $\lambda$. This kind of process is called an ``arithmetic point process". The largest scalar for which the process follows this property is called the ``span" of the process.

The deterministic point process described before is the simplest example of this (if we scale its time line down by the span, events happen at every integer in this time line). Such processes play a special role in the result we're exploring in this paper, as we'll see in the next section.

\subsection{Blackwell's theorem}
Now that we've defined the special case of arithmetic point processes, we're ready to state the core theorem we're setting out to prove.
Loosely, it says that if you go ``well into the lifetime" of a non-arithmetic renewal process and count the number of events inside an arbitrary interval of size $u$, you will find on average $\frac{u}{t}$ events inside the interval, where $t$ is the mean time between events for the process. If the process is arithmetic, the same result holds when $u$ is a multiple of the span. 

To state it formally, recall the function $\mu(s)$, which is the average number of events from the process we can expect to see from the start of the process until time $s$ has elapsed. The average number of events in an interval of size $u$ which starts at time $m$ is then given by $\mu(m+u)-\mu(m)$. We can now state the result of Blackwell's renewal theorem. A proof of this is provided in section \ref{final_proof}.

\begin{theorem}[Blackwell's renewal theorem]\label{conject:events}
For a non-arithmetic renewal process with inter-arrival times given by the i.i.d. sequence $(T_i)$ with $\Bbb E(T_i)=t$, the average number of events lying in an interval started at a large time, $m$ and of size $u$ is given by:
$$\lim_{m \to \infty}(\mu(m+u)-\mu(m))=\frac{u}{t}$$
For an arithmetic renewal process, we get the same result when $u$ is a multiple of its span, $\lambda$.
\end{theorem}

This is the same conclusion we were expecting in section \ref{sectn:count} (with the buses), with the added condition of going `well into the lifetime' of the process avoiding preferential treatment of any of the renewals (like we also alluded to there) as we will see in the next section.

In general, the literature treats arithmetic processes as ``encumbrances", special cases to be treated separately \footnote{See for example section 11.1 of \cite{feller} ``the formulation of the renewal theorem is encumbered by the special role played by distributions concentrated on the multiples of a number $\lambda$"}. We'll take the opposite approach here, taking the simplest arithmetic process (the deterministic point process explained before) and using it as the basic template over which all other point processes are built. Note that for a deterministic point process, the theorem above holds if we start our observation window at at any arbitrary time and we don't necessarily have to go ``well into its lifetime". 

\section{Why go well into the lifetime}\label{why_well_lifetime}
In this section, we'll explore what going ``well into the lifetime" of a renewal process is doing exactly for the renewal theorem and why we need the special-casing for arithmetic renewal processes. The method we'll use to prove this will be similar to the method we'll use to prove Blackwell's theorem in section \ref{final_proof}.

\begin{proposition}\label{prop:blackwell_unif}
If we take a non-arithmetic renewal process and take a time stamp, $\vartheta$ ``well into its lifetime" ($\vartheta \to \infty$), the distribution of the time from this time stamp to the next event from the process, conditional on the length of the interval (intervals are formed by successive events of the process) our time stamp, $\vartheta$ lies in being $T_{M+1}=t_{M+1}$, is uniform between $0$ and $t_{M+1}$.
\end{proposition}

\begin{proof}
Consider some time $\vartheta$ in the time line of the process where we start our observation interval. Suppose the index of the point right before the time stamp $\vartheta$ is $M$ (in other words, this event is the $M$-th event in the process, or that $M = N(\vartheta)$). This makes the index of the event right after the time stamp $\vartheta$, $M+1$ (visualized in figure \ref{fig:pt_process_basic}).

\begin{figure}
  \includegraphics[width=0.8\linewidth]{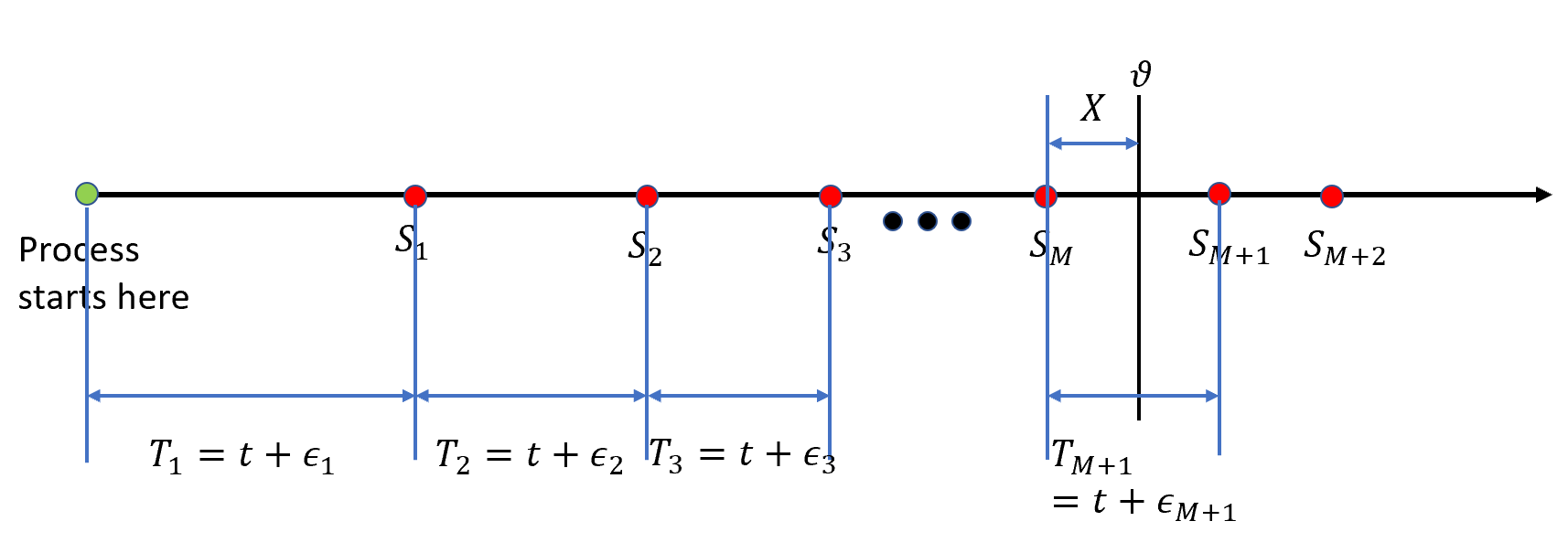}
  \caption{Time until next event when we're well into the lifetime of a point process.}
  \label{fig:pt_process_basic}
\end{figure}

\fbox{Part 1: Map to a deterministic point process in a conditional world}
Now, let's condition on the time between event at $S_M$ and the one at $S_{M+1}$ being $T_{M+1} = t_{M+1}$. Let's convert our process into a deterministic point process, with times between all events being $t_{M+1}$ (within this conditional world).

The time between the start of the process and event \#$1$ is $T_1$ (per assumptions of i.i.d. inter-arrivals for a renewal process, its distribution is the same even in the conditional world), which is not $t_{M+1}$. So, we move event \#$1$ in a way that it does become $t_{M+1}$. To do this, we move it back by amount: $\eta_1 = T_1-t_{M+1}$ (if $\eta_1$ is negative, we end up moving it forward). To preserve other properties of the process, we move not just the event but the entire time line. So, all subsequent events that happened after event \#$1$ also move by this same amount. Hence, event \#$M$ moves and will change the $X$ in figure \ref{fig:pt_process_basic}, the time between event \#$M$ and the start of our observation interval (at $\vartheta$). We want to avoid this, so we move the start of our observation window back by $\eta_1$ as well. Once we do this, we ensure that $X$ doesn't change. Now, the time between the start of our process and event \#$1$ has become $t_{j+1}$. We can repeat this for event \#$2$, \#$3$ and so on until event \#$J$. When we get to event \#$J+1$, there is no need to move it since the time between \#$J$ and \#$J+1$ is already $t_{j+1}$. And for events after \#$J+1$, we simply move them without moving the start of our observation interval (since they don't affect the random variable $X$ anymore).

Once we're done with this for all points, we end up with a deterministic point process in the conditional world, where the time between arrivals is a deterministic $t_{M+1}$. But the start of our observation interval is now at $\vartheta-\sum\limits_{k=1}^M \eta_k$. Hence for any realization of the original point process, there is a corresponding realization for a deterministic point process where we start the interval at time $\tau = \vartheta -\sum\limits_{k=1}^M \eta_k$.

Now, we have two point-processes. The original one (stochastic arrivals but observation window starts at a deterministic $\vartheta$) and the modified one which we ensured is deterministic (deterministic arrivals but observation window starts at a stochastic time, $\tau$). Observe the following things about the two processes:

\begin{itemize}
\item{Any realization of the original process maps bijectively to a realization of the deterministic process.}
\item{The probability of the realization of the original process is the same as the realization of the deterministic process (since it involves all the same random numbers).}
\item{The variable, $X$ we're interested in stays the same between the two processes.}
\end{itemize}

Hence, we can make inferences about the $X$ in the deterministic process in-lieu of the original one (within the conditional world where $T_{M+1} = t_{M+1}$).

\fbox{Part 2: Uniformity achieved when $\vartheta$ grows in deterministic process} Now, consider what happens for the deterministic point process as $\vartheta$ grows. This will cause $M$ to grow without bound as well (observe $M=N(\vartheta)$ and refer to lemma 4.3.2 of \cite{ocw_renewal}) and the variance of the term, $\sum\limits_{k=1}^M \eta_k$ will explode. Then, we take the conditional distribution of the random variable $\vartheta-\sum\limits_{k=1}^J \eta_k$ inside an interval of size $t_{M+1}$, which is just a realization from the random variable $T_i$, the inter-arrival times of the renewal process. Hence, this interval must be finite.

By the central limit theorem, $Z := \vartheta - \sum\limits_{k=1}^M \eta_k$ will tend towards a Gaussian distribution with some mean and a variance that grows as $\vartheta$ increases and hence $M$ becomes larger and larger (since the $\eta_k$ are i.i.d. per the assumption of the renewal process). Further, we can re-scale the axis of our deterministic point process in a way that the times between events becomes $1$ in our conditional world (divide all times by $t_{M+1}$). This re-scaling converts the $Z$ to another random variable, $Y$,

$$Y = \frac{\vartheta - \sum\limits_{k=1}^M \eta_k}{t_{M+1}}$$

 which is also Gaussian and whose variance also goes to $\infty$ as $\vartheta \to \infty$. This ensures that the arrivals in our re-scaled deterministic point process are placed at the positive integers. 

Now, all we need to show is that $Y-\lfloor Y \rfloor \sim U(0,1)$ as the variance, $\sigma_1^2$ of $Y$ goes to $\infty$ and we're done (the fact that we're still in the conditional world doesn't matter because the result doesn't depend on the variable we're conditioning on, $t_{M+1}$ so we can just integrate it out). We need to be careful though, since an obvious exception occurs when the $Y$ can only take integral values since $Y-\lfloor Y \rfloor$ is identically $0$ in that case. 

\fbox{Part 3: Proving uniformity}
The theorem that follows (\ref{thm:sum_uniform}) shows that only under severe restrictions on the support of the $T_i$ does the condition: $Y-\lfloor Y \rfloor \sim U(0,1)$ fail to hold \footnote{It is heavily inspired by \href{https://math.stackexchange.com/questions/4157329/if-x-is-gaussian-prove-that-x-lfloor-x-rfloor-sim-u0-1-as-its-variance}{this post}}.
\end{proof}

\begin{theorem}\label{thm:sum_uniform}
Let $(T_k)$ be a sequence of i.i.d. real valued random variables and define for each $m\in {\Bbb Z}$: $\gamma_m = {\Bbb E} \left( e^{2\pi i m T_1} \right)$. Set $S_n=T_1+\cdots T_n$. Then the following are equivalent:
\begin{enumerate}
\item{ The law of $S_n  \ {\rm mod}\ 1$ converge in distribution to  $\ { U}([0,1))$ }
\item{ $|\gamma_m|<1$ for every $m\in {\Bbb Z} \setminus \{0\}$.}
\item{For every $m\in {\Bbb Z} \setminus \{0\}$, $\theta\in [0,1)$: $\ {\rm support} (m X_1) \not\subset {\Bbb Z}+\theta$}
\end{enumerate}
\end{theorem}
\begin{proof}
By the i.i.d condition
$${\Bbb E} \left( e^{2\pi i m S_n} \right)
   = {\Bbb E} \left( e^{2\pi i m T_1} \right)^n = \gamma_m^n$$
Thus, if $g$ is a 1-periodic trigonometric polynomial, then
${\Bbb E}(g(S_n))\to \int_0^1 g(t)dt$ whenever $|\gamma_m|<1$ for every non-zero $m$. To see this, consider:

$$g(t) = \sum_m c_m e^{2 \pi i m t}$$

then,

$${\Bbb E}(g(S_n)) = \sum_m c_m \gamma_m^n  \xrightarrow[n \to \infty]{} c_0 \gamma_0  = c_0$$

Conversely if for some non-zero $m$, $\gamma_m=e^{i\phi}$ then the convergence does not take place for
$g=\exp(2\pi i m x)$. 

As trigonometric polynomials are dense in the 1-periodic functions we get that $1\Leftrightarrow 2$.
To see that 2 and 3 are equivalent, simply note that for non-zero $m$
 $${\Bbb E} \left( e^{2\pi i m T_1} \right) = e^{2 \pi i  \theta}$$
iff $mT_1 \in {\Bbb Z}+\theta$ almost surely.
\end{proof}

Note that the condition, $mT_1 \in {\Bbb Z}+\theta$ corresponds to shifted arithmetic point processes. For an arithmetic process we require:

$$T_1 \in \lambda I \;\; | I \in \Bbb Z$$

When the span, $\lambda$ is a rational number, we can express it as $\lambda = \frac{a}{m}I$ where $a$ and $m$ belong to $\Bbb Z \setminus \{0\}$. This means,

$$m T_1 = a I$$

and so $mT_1$ can only take integral values since $a$ and $I$ are both integers. Now, it we shift such a $T_1$ by some real number $s$, the integer part of $s$ won't change the support of $mT_1$, but the fractional part ($\theta$) will instead make the support: ${\Bbb Z}+\theta$.

We reached this conclusion when the span, $\lambda$ is a rational number, but since the rationals are a dense subset of the reals, we can make the claim for any real $\lambda$ to an arbitrary degree of precision. \footnote{Heavily inspired by \href{https://math.stackexchange.com/a/4159649/155881}{this page}.}

\section{Choosing where to start}
We showed in the previous section that the time from the start of a random window (well into its lifetime) to the next event in the process (conditional on the interval containing the window being a certain size) is uniform over that interval size for non-arithmetic renewal processes. Further, this crucial property (henceforth the ``uniformity property") isn't secured for arithmetic renewal processes with the same convenient strategy. And this is the source of the `encumbrance' Feller refers to.

So, we can consider picking a slightly less convenient strategy for deciding the start of our observation interval. But, we can ensure that such strategies secure the uniformity property for all point processes, not just non-arithmetic renewal processes. All the while, we'll know that this strategy can be replaced with the more general `going well into the lifetime' strategy with the same effect for non-arithmetic renewal processes. 

If we could choose our window start uniformly over the entire domain of our renewal processes, this would certainly ensure the uniformity  property. But the problem is that the domain of our processes is $[0,\infty)$ which is an infinitely large window. And, there is no such thing as a random number uniform over an infinite domain. This is bad news, but we can still approximate a random number like this as closely as we want. 

An obvious way to do this is to start our observation window at uniform random number with a larger domain, but there are many others. Another could be to use an exponential random variable with a vanishing rate. The first approach seems simple enough, so that's the one we'll go with.

Now, we show formally that this approach for choosing the starting window drawn from a large uniform distribution produces the same result as starting an observation window ``well into the lifetime" of the process, i.e. conditional on the inter-arrival interval that contains the start of our time window being $t$, the distribution of the time from the start of the window to the first event is uniform over $0$ to $t$. 

\subsection{A large uniform number}

\begin{proposition}
If we take a large uniform random number, $V = \lim_{\vartheta \to \infty} \vartheta U$, and start an observation interval at this time-stamp, the distribution of the time until the next event in the point process, conditional on the interval (time window between successive events in the process) that contains that time-stamp being of size $t_{M+1}$ is uniform over $(0,t_{M+1})$.
\end{proposition}
\begin{proof}
If we take a uniform random number and consider the conditional distribution that it will be within a sub-interval completely contained within its support, its easy see (or prove with Bayes' theorem) that this conditional distribution is uniform over the sub-interval. 

Let $S_M$ be the time of the event right before time $\vartheta$. If our uniform random number between $0$ and $\vartheta$ lies anywhere from $0$ to $S_M$, the sub-intervals are completely contained within it. Hence, the result above concerning the sub-intervals of a uniform random number applies and the result holds.

However, if the uniform random number lies anywhere from $S_M$ to $\vartheta$ (red region in figure \ref{fig:unif_start}), the interval from $S_M$ to $S_{M+1}$ lies only partially inside the support of our uniform random number. In this case, the premise of our proposition is violated. However, the probability of this will be:

$$P(S_M < V < \vartheta) = \frac{\vartheta-S_M}{\vartheta} < \frac{t_{M+1}}{\vartheta}$$

Assuming samples from the inter-arrival distribution, $T_i$ are always finite, this probability tends to zero as $\vartheta \to \infty$. This means that as $\vartheta \to \infty$, our uniform random number $V$ will lie within the interval $(0,S_M)$ almost surely. Which is the region where the result we desire holds.

\begin{figure}
  \includegraphics[width=0.8\linewidth]{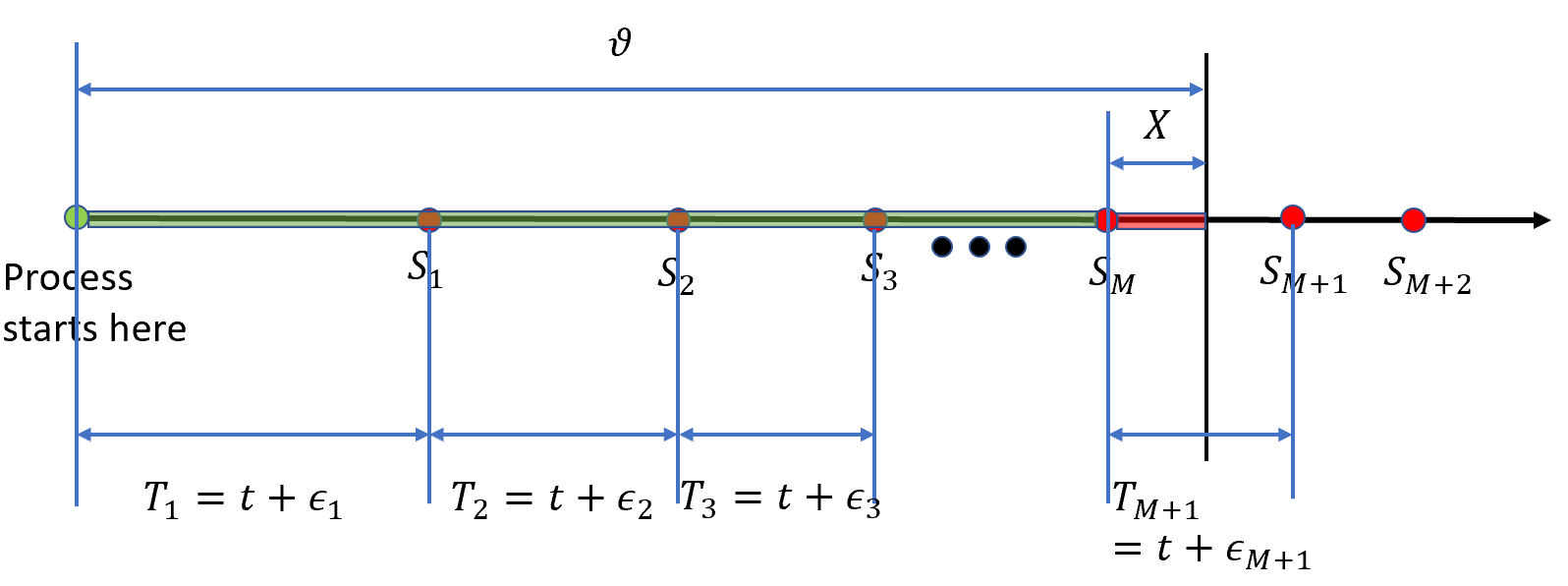}
  \caption{The time until next event in the point process.}
  \label{fig:unif_start}
\end{figure}

\end{proof}

\subsection{Large uniform deferred}
\begin{proposition}\label{prop_unif_deferred}
If we take a large uniform random number, $V = \lim_{\vartheta \to \infty} \vartheta U$, and start an observation interval at time-stamp, $V+c$ where $c$ is a constant, the distribution of the time until the next event in the point process, conditional on the interval (time window between successive events in the process) that contains that time-stamp being of size $t_{M+1}$ is uniform over $(0,t_{M+1})$.

\end{proposition}
It is not hard to extend the proof from the previous sub-section and show that even if we add a constant term to the end of the uniform interval, starting it at $V+c$ instead of at $V$, the result around uniformity still holds (which is what the proposition above is saying).

\section{Deterministic point process}
\label{deterministic}

\subsection{Conditional uniform is a way}
In this section, we'll prove Blackwell's renewal theorem when following the `large uniform' strategy for picking the start of the observation interval for the deterministic point process. To do this, we'll need some results which we show first. 

\begin{lemma}\label{prop:unif_minus_1}
If $U$ is a uniform random number over $[0,1)$ and $c \in \Bbb R$ is a deterministic scalar then we have:
$$\Bbb E(\lfloor c-U \rfloor) = c-1$$
\end{lemma}
\begin{proof}

The expected value is
$$\mathbb{E}(\lfloor c-U \rfloor)=\int_0^1 \lfloor c-u\rfloor du$$
 Let us assume that $n<c<n+1$. We have that:
 
  $$\lfloor c-u \rfloor =
\begin{cases}
    n, & \text{if } c-u\geq n\\
    n-1,              & c-u<n
\end{cases}  
  $$

  
  Therefore,
  
$$\mathbb{E}( \lfloor c-U \rfloor) =nP(u\in[0,c-n])+(n-1)P(u \in (c-n,1])$$
$$=n(c-n)+(n-1)(1-c+n)$$
$$=nc-n^2+n-nc+n^2-1+c-n=c-1$$ \footnote{Note: Heavily inspired by the post \href{https://math.stackexchange.com/questions/4123742/prove-that-ec-u-c-1}{here}}.
\end{proof}

\begin{lemma}[Adding zero mean noise]\label{thm:add_white_noise}
If $\eta$ is a random variable with mean $0$ and that's all we know about it:

$$\Bbb E( \lfloor c+\eta-U \rfloor) = c-1$$

\end{lemma}
\begin{proof}
First, we condition on $\eta$ to get (using proposition \ref{prop:unif_minus_1}):

$$E([c+\eta-U]|\eta) = c+\eta-1$$

Now, take expectation again

$$E([c+\eta-U]) = E(E([c+\eta-U]|\eta)) = E(c+\eta-1) = c-1$$
\end{proof}

\begin{proposition}[Renewal theorem for a deterministic point process (perturbed)]\label{prop:deterministic_perturb}
For a deterministic point process, which starts at time $0$ and events happen every $t$ interval (where $t$ is a non-stochastic scalar), if we start an observation interval at $V+\eta_1$, where $V = \lim_{m \to \infty} m U$ is a large uniform random number and $\eta_1$ is random variable with mean zero, and make the size of our observation interval $u+\eta_2$ where $u$ is a scalar and $\eta_2$ is another random variable with mean zero, the expected number of events falling into our interval will be $\frac{u}{t}$.
\end{proposition}
\begin{proof}

Note that the interval size now becomes $u+\eta_2-\eta_1$. Let's first condition on $\eta_1$ and $\eta_2$. Everything that comes next is within this conditional world.
Let $X$ be the time from the start of the interval to the very first event. By proposition \ref{prop_unif_deferred}, $X$ is a uniform random variable between $0$ and $t$. As before, we count the first event and then the remaining events in what's left of the interval. Define $u' = u+\eta_2-\eta_1$

$$N = 1 + \left \lfloor \frac{u'-X}{t} \right \rfloor =1+\left\lfloor \frac{u'}{t}-\frac{X}{t} \right\rfloor = 1+\left\lfloor \frac{u'}{t}-U \right\rfloor $$

Where $U$ is a uniform random variable between $0$ and $1$. Taking expectation and expressing the conditioning that got us this result, we get (from lemma \ref{prop:unif_minus_1}):

$$E(N | \eta_1, \eta_2) = \frac{u'}{t} = \frac{u+\eta_2-\eta_1}{t}$$

By the law of total expectation,
$$E(N) = E(E(N | \eta_1, \eta_2)) = \frac{u+E(\eta_2-\eta_1)}{t} = \frac{u}{t}$$

\end{proof}

\begin{corollary}[Renewal theorem for a deterministic point process]\label{prop:conjecture_deterministic}
For a deterministic point process, which starts at time $0$ and events happen every $t$ interval (where $t$ is a non-stochastic scalar), if we start an observation interval at a large uniform random number, $V$ (with size of the interval being $u$), the expected number of events falling into our interval will be $\frac{u}{t}$.
\end{corollary}
\begin{proof}

A trivial consequence of proposition \ref{prop:deterministic_perturb}, setting $\eta_1 = \eta_2 = 0$.

\end{proof}

\subsection{Conditional uniform is the only way}

\begin{proposition}
If $U$ is a r.v. with support over $(0,1)$ that satisfies:
$$\Bbb E(\lfloor c-U\rfloor)=c-1$$
 for all $c$ then it must be uniform.
\end{proposition}
\begin{proof}  If $n<c<n+1$,
$$c-1=\mathbb{E}(\lfloor c-U \rfloor)=nP(0<U\leq c-n)+(n-1)P(c-n<U \leq 1])$$
In particular, for any $0<c<1$ ($n=0$)
$$c-1=\mathbb{E}( \lfloor c-U \rfloor)=-P(c<U \leq 1)$$
From this, it follows that 
$$c=1-P(c<U\leq 1)=P(0< U\leq c) $$

Showing $U$ is uniform \footnote{Note: this proof is inspired by the post \href{https://math.stackexchange.com/a/4152660/155881}{here}}.
\end{proof}

\section{Extending to general point processes: proof of renewal theorem}\label{final_proof}
This whole section is the proof to Blackwell's renewal theorem we promised.

\subsection{Defining the variables}

We showed in corollary \ref{prop:conjecture_deterministic} that theorem \ref{conject:events} holds for a deterministic point process (under a slightly stricter condition on where we start our observation window). We now want to generalize that result. It is clear that we can go from a deterministic point process to a general point process by adding random noise to the inter-arrival times. And it seems intuitive that if this noise we've added has zero mean, it shouldn't change the conclusion (since it should sometimes increase the number of events falling into our interval and sometimes decrease them with the two effects canceling out in expectation). This is the approach we'll follow in this section, except we'll convert our general point process into a deterministic one, adding on assumptions to the $T_i$ as required along the way.

Recall in the deterministic point process, we had $\epsilon_i=0$ and events simply happen every $\Bbb E (T_i) = t$ interval. We showed that the conjecture holds for this point process in corollary \ref{prop:conjecture_deterministic}. We will now try to transform the general point process into the deterministic one.

We go from the start of it to some time $V$ (a large uniform random number) and take an interval of size $u$ starting there. The end of the observation interval is labeled $u_2$ and the start of it is labeled $u_1$.  For now, we simply have $V=u_1$, but this will change in some other scenarios. Refer again to the terminology defined in section \ref{terminology}. Figure \ref{fig:pt_process_observed} is identical to figure \ref{fig:pt_process_v0}, apart from two new random variables, $X$ and  $Y$ that we'll need. $X$ is the time elapsed between the start of the observation period, $u_1$ and the event immediately preceding it, while $Y$ is the same thing for the end of the observation interval, $u_2$. Note that requirements we have so far on the $T_i$'s is that:

\begin{itemize}
\item{They have the same mean, $t$\footnote{This requirement can be relaxed by noting that even if the $T_i$'s can have different means from a list of potential candidates, we have no way of knowing in advance which of those candidates will come into play.}. }
\item{We required them to be independent in proposition \ref{prop:blackwell_unif}.}
\end{itemize}

So far, we haven't required them to be identically distributed (apart from having the same mean). Sadly, we'll need to use these conditions as the proof progresses, but will provide an outline of a means to potentially get rid of it.

\begin{figure}
  \includegraphics[width=0.8\linewidth]{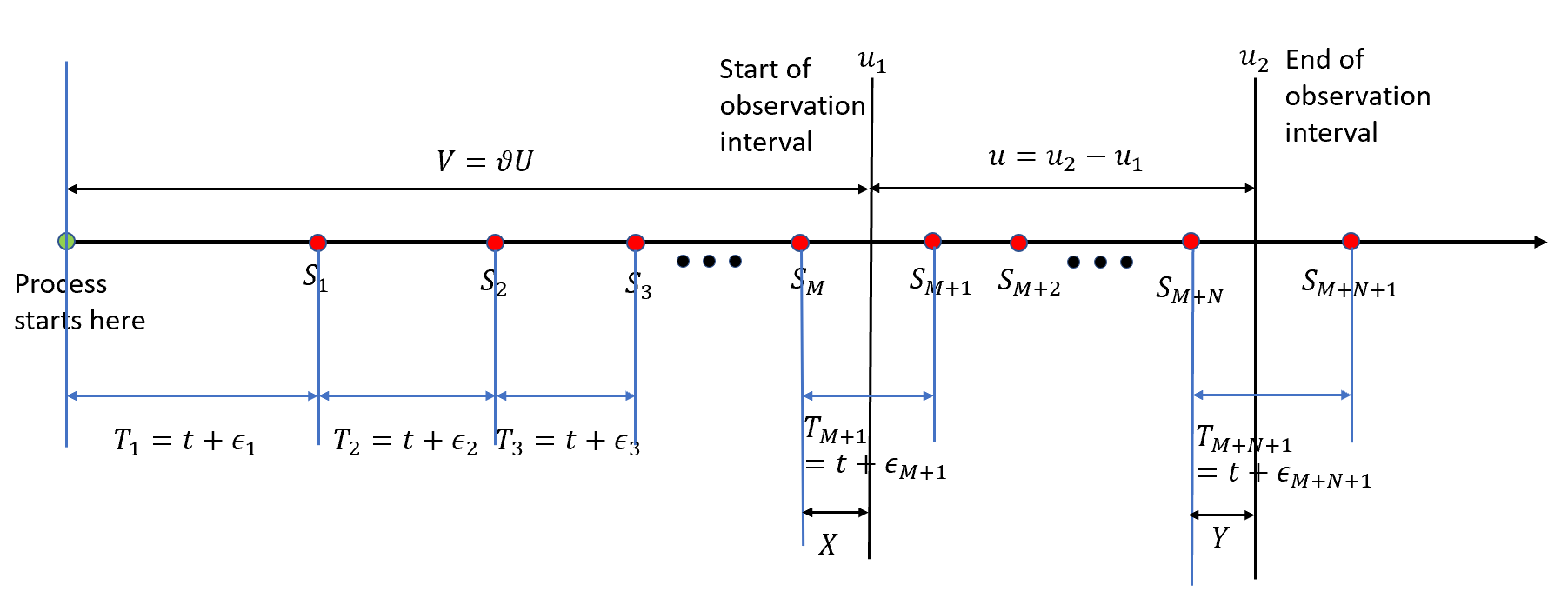}
  \caption{A generic point process with some observation window and associated random variables.}
  \label{fig:pt_process_observed}
\end{figure}

\subsection{Converting to deterministic}

We said we wanted to take the random point process and map it back somehow to the deterministic point process where the $\epsilon_i$'s are all zero. So, let's just make them all zero. In doing this, we'll be modifying various aspects of our existing point process (the times at which the events happen as well as the observation window) and converting it into a modified one. We'll show that the expected number of events in the modified point process within its modified observation window is the same as the original point process within its observation window.

We start with the very first event, (at time $S_1$). This event happens at a random time since the start of the process. If we remove the randomness by shifting the entire time line, we'll need to move it (event at $S_1$) and everything that happened after it back by $\epsilon_1$. This means the first event, the second event (at $S_2$), the third event (at $S_3$) and so forth, until infinity all move back by $\epsilon_{1}$. Once we do this, the time interval between the start of the process and the first event becomes a deterministic, $t$. But, this might also cause the events lying inside our observation interval to change. To prevent this, we move the entire observation interval (the start of the interval, $u_1$ and the end of it, $u_2$) back by $\epsilon_1$ as well. Since all the points as well as the observation interval moved by the same amount, the events inside this interval (as well of course, as the number of them), its size and every other aspect of the data falling within it (like the variables $X$ and $Y$) stays the same.

Now we don't move the event at $S_1$ anymore from here on out. To remove the randomness in the second interval ($T_2$), we repeat this process by moving the events at $S_2$, $S_3$ and so on as well as the observation interval goal posts ($u_1$ and $u_2$) back by $\epsilon_2$, and this makes the time elapsed between the events at $S_1$ and $S_2$, $t$ as well.

By the time we get to the event at $S_M$, we've moved it back by the amount: $\epsilon_1 + \epsilon_2 + \dots + \epsilon_{M}$. Here, we run into a problem. So far, we've been moving keeping the event at the start of the interval fixed, while moving the event at the end of the interval as well as all subsequent events and the goal-posts ($u_1$ and $u_2$) defining the interval. Moving everything like this by the same amount ensured we didn't disturb what went on inside the interval. Note that when we moved the event at $S_M$ back (and the start of our observation interval, $u_1$ with it), the event at $S_{M-1}$ wasn't moving. So, $u_1$ was potentially moving towards it (if $\epsilon_M$ were positive). There was however no danger that $u_1$ would actually touch the event at $S_{M-1}$, much less cross it, since it was guaranteed to stay ahead of the event at $S_M$ (since $S_M$ was moving with it) and $S_M$ was guaranteed to stay ahead of $S_{M-1}$. This  guarantee similarly applied to all events before $S_{M-1}$.

Now however, when we reach event the event at $S_M$ and move event at $S_{M+1}$ back, we lose this guarantee. There is a chance that as $u_1$ moves backward towards $S_M$, it might touch it and actually cross it. To avoid this, we don't move $u_1$.

So, when we move \#$M+1$ and all subsequent events back by $\epsilon_M$, we don't move $u_1$, instead moving only the end of the interval, $u_2$ back by this amount. There is now however a risk now that the event at $S_{M+1}$ crosses $u_1$ and goes outside our interval. If we denote by $X$, the random variable representing the time elapsed between $S_M$ and $u_1$ (shown in green in figure \ref{fig:pt_process_observed}), the probability that the event at $S_{M+1}$ will move outside our observation window in the new process is $P(X>t)$. This is because the time between $S_M$ and $S_{M+1}$ is going to become $t$ after the movement and if $X$ was greater than $t$, it'll have to mean that $S_{M+1}$ crossed the boundary $u_1$. 

After we make the times between $S_M$ and $S_{M+1}$ as well as $S_{M+1}$ and $S_{M+2}$ deterministic $t$, the time between events $S_M$ and $S_{M+2}$ will become $2t$. However, we know that $\epsilon_M$ can't be less than $-t$. So, it is guaranteed that $S_{M+2}$ will be inside the modified interval in the new point process, just as it was inside the interval for the original one. 

We now repeat the process for all events until the one at $S_{M+N-1}$, each time moving only the end of the interval, $u_2$ back by the same amount as well. It's easy to see that none of these operations will cause any events to enter or leave the interval. But finally, when we move $S_{M+N+1}$ back by amount $\epsilon_{M+N+1}$, we don't do the same for $u_2$. Now there is a chance that $S_{M+N+1}$ will enter our modified interval. Using similar reasoning as before, the probability of this happening is $P(Y>t)$ where $Y$ is the time from $u_2$ to $S_{M+N+1}$ (also shown in figure \ref{fig:pt_process_observed}). 

We then move all the events after $S_{M+N+1}$ back by the amounts that cause the time between them and the preceding event to become $t$. And in doing so, we don't change the number of events within our interval any further.

After doing all of this, we've transformed our original point process with random inter-arrival times to a deterministic point process with deterministic inter-arrival times ($t$ each). We know that events at $S_{M+1}$ through $S_{M+N}$ will be inside the interval for both the original and modified processes for sure. But the event at $S_{M+1}$ might go out of the interval with probability $P(X>t)$ while event at $S_{M+N+1}$ might come inside the interval with probability $P(Y>t)$. If these two probabilities are equal (which would mean $X$ and $Y$ are identically distributed), the expected number of events inside the interval for the original process will be the same as the expected number of events in the modified process with its modified interval. To see why this must be the case, we use proposition \ref{prop:blackwell_unif} which says that conditional on the interval our observation window lying in being $t_k$, the distribution of $X$ is uniform between $0$ and $t_k$. It is trivial to see that this must be true for $Y$ as well (since we only relied on $u_1$ being large in the proposition for the conditional uniformity to hold and $u_2>u_1$. Which means that $X$ and $Y$ are both mixtures of uniform distributions (with the minimum value of these uniform distributions being $0$ and the maximum value drawn from the distribution of $T_i$). This tells us that $X$ and $Y$ have the same distribution and so, the two probabilities in question ($P(X>t)$ and $P(Y>t)$) are equal as required\footnote{Note that we relied here on the observation intervals, $T_i$ being identically distributed. This could have been avoided by noting that even if the $T_i$ could follow different distributions (from a list of candidate distributions), we have no way of knowing which of those candidate distributions will apply to the interval containing $u_1$ and which to the one containing $u_2$ in advance.}. 

If we can now prove that the modified deterministic point process we end up with has $\frac{u}{t}$ events on average inside the modified interval, it will show that the original process has the same number of events on average in its interval and complete the proof of Blackwell's renewal theorem.

To recap, in the modified process, we start the interval not at $V$ but had to move it back by $\epsilon_1 + \epsilon_2 + \dots \epsilon_M$. In other words, the start of the interval, $u_1$ becomes:

$$u_1 = V - \sum\limits_{i=0}^M \epsilon_i$$

And the end of the interval, $u_2$ becomes:

$$u_2 = V + u - \sum\limits_{i=0}^{M+N} \epsilon_i$$

By proposition \ref{prop:deterministic_perturb}, we get that the expected number of events in this interval is $\frac{u}{t}$ as desired for the modified process.

Note that this mapping the general point process to a deterministic one is the same trick we applied in section \ref{why_well_lifetime} and we can justify conclusions about the expected events in the observation window of deterministic point process applying to the original point process for the following reasons:

\begin{itemize}
\item{Any realization of the original process maps bijectively to a realization of the deterministic process.}
\item{The probability of the realization of the original process is the same as the realization of the deterministic process (since it involves all the same random numbers).}
\item{The expected number of events falling into the respective intervals of the two processes stay the same.}
\end{itemize}

And this completes our proof for Blackwell's renewal theorem.

\section*{Acknowledgements}
I'd like to thank \href{https://math.stackexchange.com/users/355946/h-h-rugh}{HH Rugh}, \href{https://math.stackexchange.com/users/121671/oliver-diaz}{Oliver Diaz} and \href{https://math.stackexchange.com/users/758600/paresseux-nguyen}{Paresseux Nguyen}, users on the website, \href{https://math.stackexchange.com/}{math.stackexchange} who provided a lot of the proof's in this paper (I linked the posts where ever I was inspired by their answers). Sudarsan V Ranganathan pointed me to excellent resources on renewal theory so I could build some background. He also pointed out problems with an earlier version of the proof. Finally, I had some discussions with Adam Gustafson on point processes and what we can say about their failure rates, which motivated and triggered this exploration.

\begin{appendices}
\section{Python simulators}
\subsection{Simulator}\label{py:code}
Here, we provide some Python code demonstrating that if we have a random variable which can take value $0$ and $20$ with equal probabilities and have it as the distribution of the $T_i$'s, the average number of events falling into an interval of size $1$, started at a large uniform random number is $0.1$, as expected by the renewal theorem. 

\lstset{language=Python}
\lstset{frame=lines}
\lstset{caption={Simulator demonstrating the renewal theorem}}
\lstset{label={lst:simulator}}
\lstset{basicstyle=\footnotesize}
\begin{lstlisting}
import numpy as np

def sim_bimodal():
    catches = 0
    for _ in range(50000):
        j = np.random.uniform()*1000
        t_i = 0
        while t_i < j+100:
            if np.random.uniform() < 0.5:
                t_i += 0
            else:
                t_i += 20
            if j < t_i and t_i < j+1:
                catches += 1
    print(catches/50000)
\end{lstlisting}

\subsection{$E\left(\lfloor c-U \rfloor \right) = c-1$}
The following Python simulation shows $E\left(\lfloor c-U \rfloor \right) = c-1$

\lstset{language=Python}
\lstset{frame=lines}
\lstset{caption={Simulator demonstrating the conjecture}}
\lstset{label={lst:c-u}}
\lstset{basicstyle=\footnotesize}
\begin{lstlisting}
import numpy as np
# For example,
c = 3.2
np.mean([int(i) for i in c-np.random.uniform(size=10000)])
\end{lstlisting}

\section{Distribution of residual life}\label{large_unif}

Consider starting our observation interval at a large uniform random number, $V=mU$ (with $m$ being a large number like before and $U$ being a standard uniform random number between $0$ and $1$). We have the following proposition describing the time from the start of the interval to the first event thereafter:

\begin{proposition}
If we take a deterministic point process with inter-arrival time $t$ and start an observation window at a large uniform random number, the distribution of the time from the start of the window to the next event is uniform over $(0,t)$.
\end{proposition}

If $m \in \Bbb Z$, its easy to see that the proposition holds even for small $m$, so we consider the case $m\not\in\Bbb Z$. Further, without loss of generality, we can scale the time axis so that the deterministic events happen at integers (intervals of $1$, by scaling down by $t$). The proposition then reduces to: $mU-\lfloor mU \rfloor \sim U(0,1)$.

Then $Y_m=mV\sim U(0,m)$ and $Z_m=\lceil Y_m\rceil-Y_m\in[0,1)$. The distribution function of $Z_m$,

\begin{align*}
F_{Z_m}(x)&=P(\lceil Y_m\rceil-Y_m\le x)\\&=\begin{cases}0,&x\le0\\P(Y_m\in[k-x,k],k\in\Bbb N_{\le\lfloor m\rfloor}),&0<x\le \lceil m\rceil -m\\P(Y_m\in[k-x,k],k\in\Bbb N_{\le\lfloor m\rfloor})+P(\lceil m\rceil-x\le Y_m\le m),&\lceil m\rceil -m<x<1\\1,&x\ge1 
\end{cases}\end{align*}

Note that:

\begin{align*}
&P(Y_m\in[k-x,k],k\in\Bbb N_{\le\lfloor m\rfloor})+P(\lceil m\rceil-x\le Y_m\le m)\\&=\sum_{k=1}^{\lfloor m\rfloor}\int_{k-x}^k\frac{dy}m+\frac1m\int_{\lceil m\rceil -x}^mdy\\&=\frac{x\lfloor m\rfloor}m+\underbrace{\frac1m\int_{\lceil m\rceil -x}^mdy}_{\le\frac1m}
\end{align*}

so the underlined expression vanishes as $m\to\infty$ and the asymptotic behavior of the second and third cases is identical. Taking the limit as $m\to\infty$,

\begin{align*}
\lim_{m\to\infty}F_{Z_m}(x)&=\begin{cases}0,&x\le0\\\lim_{m\to\infty}\frac{x(m-\{m\})}m,&0<x<1\\1,&x\ge1 \end{cases}\\
&=\begin{cases}0,&x\le0\\x,&0<x<1\\1,&x\ge1 \end{cases}\\
&\sim U(0,1)
\end{align*}

where $\{\cdot\}$ is the fractional part function.

\begin{figure}
  \includegraphics[width=0.8\linewidth]{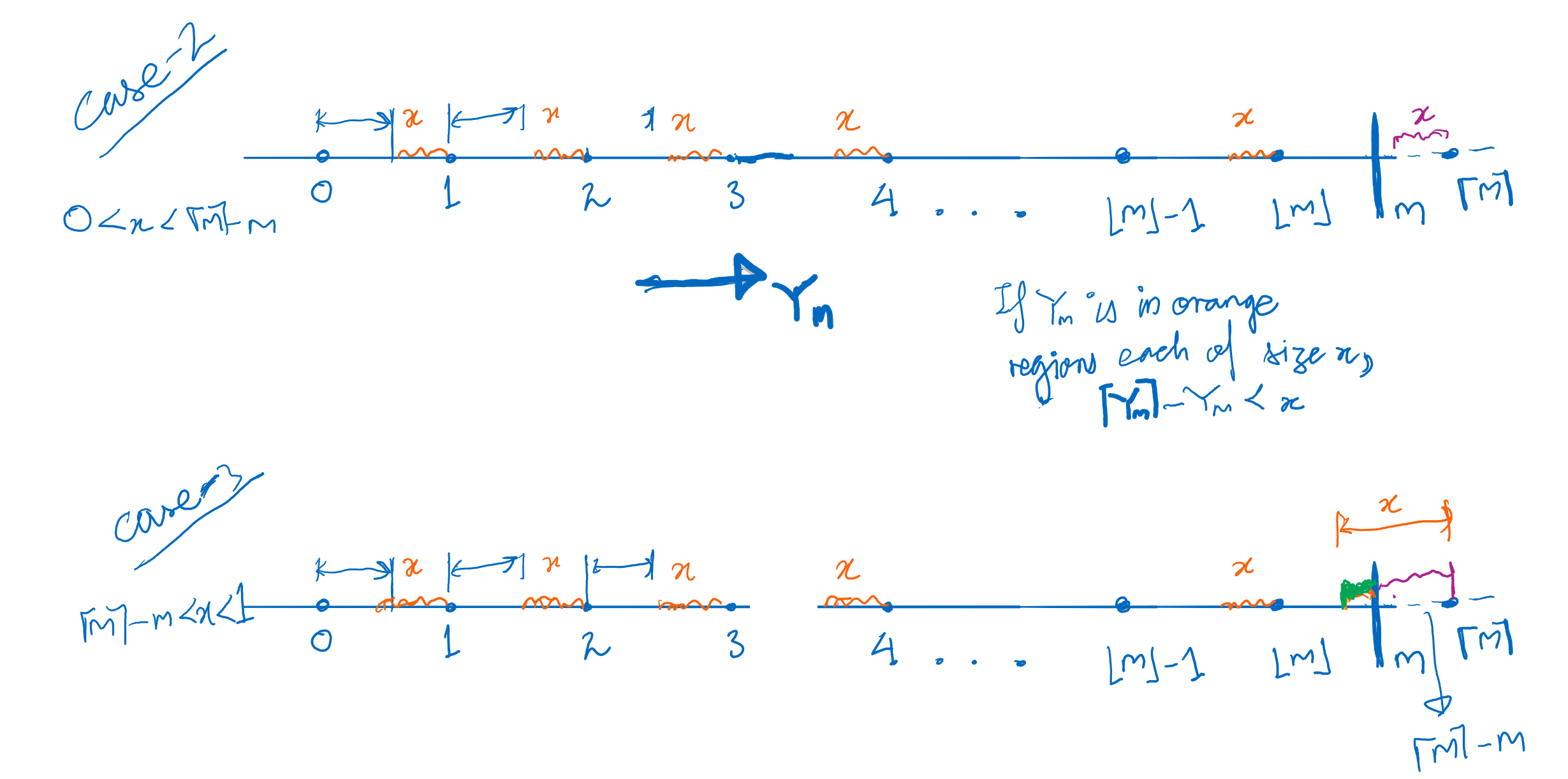}
  \caption{Time until next event when we arrive at a time which is a large uniform random number.}
  \label{fig:time_to_evnt}
\end{figure}

Refer to figure \ref{fig:time_to_evnt}. In case-2, we have to ignore the purple area. But in case-3, some of the purple area turns green and adds an additional term. The orange regions are common to both cases\footnote{A large part of this section is taken from \href{https://math.stackexchange.com/a/4139532/155881}{this page}.}. 

\section{Time until next event}
In this section, we'll explore the time from an observation point taken well into the lifetime of a point process (at $V$, a large uniform random number) to the next event immediately after. This is called the `residual time' for the point process.
Let's call $f_S(s)$ the PDF of the interarrival times of the original process. Consider figure \ref{fig:time_to_evnt_2}. Let's say the observation point lies between two events. Let's  label them \#$1$ and \#$2$. The time between these two events is $T$. What is the PDF of this $T$? We know that larger intervals are more likely to harbor the end point of $V$ inside them. And since $V$ is a large uniform, the likelihood increases linearly with the size of the interval. So, the PDF of $T$ will be proportional to $tf_T(t)$. If we include the normalizing term, this PDF becomes:

$$g_T(t) = \frac{tf_S(t)}{E(S)}$$

Where $E(S)$ is given by:

$$E(S) = \int_{t=0}^{\infty} s f_S(s) ds$$
Now, we want the distribution of $X$, the time from $V$ to event \#$2$.

\begin{figure}
  \includegraphics[width=0.8\linewidth]{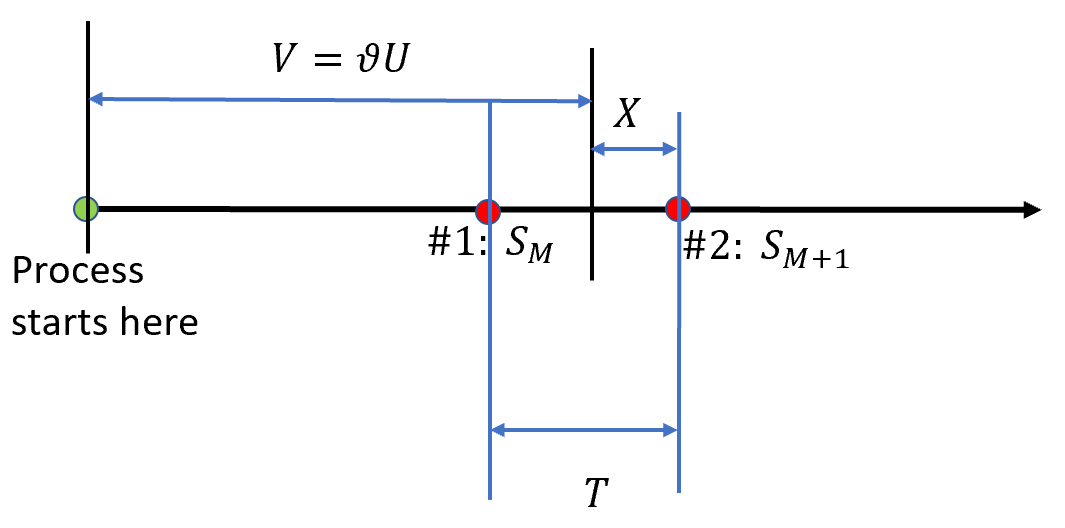}
  \caption{Residual life in a point process.}
  \label{fig:time_to_evnt_2}
\end{figure}

Consider: 

$$P(x<X<x+\delta x) = \int_{t=x+\delta x}^{\infty} P(x<X<x+\delta x | T=t) g_T(t) dt$$

$$ = \int_{t=x+\delta x}^{\infty} \frac{\delta x}{t} \frac{t f_T(t)}{E(T)}$$
$$ =  \frac{\delta x}{E(T)} P(T>x+\delta x)$$
Taking $\delta x$ to the other side and taking limits:

$$\lim_{\delta x \to 0} \frac{P(x<X<x+\delta x)}{\delta x} = \frac{P(T>x)}{E(T)}$$

$$h_X(x) = \frac{P(T>x)}{E(T)}$$

In other words, the PDF of $X$ is proportional to the survival function of $T$.

Consider the case of the Poisson process. We know here that $X$ must be exponentially distributed. And indeed, for the exponential distribution, the PDF is proportional to the survival function. In fact, since the exponential distribution is the only one that satisfies this property, we see that $X$ will have the same distribution as $S$ only for the Poisson process.

\end{appendices}
\end{document}